\documentclass{article}

\usepackage{fullpage}
\usepackage[
	colorlinks,
	linkcolor=blue,
	citecolor=blue,
	urlcolor=blue,
]{hyperref}
\usepackage{authblk}

\usepackage{enumitem}
\makeatletter
\def\namedlabel#1#2{\begingroup
	#2%
	\def\@currentlabel{#2}%
	\phantomsection\label{#1}\endgroup
}
\makeatother

\providecommand{\keywords}[1] % keyword command
{
	\noindent \small	
	\textbf{\textit{Keywords---}} #1
}

\usepackage{xcolor}
\definecolor{red}{RGB}{255,0,0}
\definecolor{blue}{rgb}{0.0, 0.4, 0.65}
\definecolor{orange}{RGB}{253,188,64}
\definecolor{green}{RGB}{154,205,50}
\definecolor{orangeplot}{RGB}{245,138,42}
\definecolor{greenplot}{RGB}{0,200,0}

\usepackage{amsmath,amsthm,amssymb}
\allowdisplaybreaks % allow page breaks inside equations with amsmath

\theoremstyle{plain}% default
\newtheorem{theorem}{Theorem}
\newtheorem{proposition}{Proposition}
\newtheorem{corollary}{Corollary}

\theoremstyle{remark}% default
\newtheorem{remark}{Remark}

\newcommand\A{{\mathcal{A}}}
\newcommand\B{{\mathcal{B}}}

\newcommand\I{{\mathcal{I}}}

\newcommand\K{{\mathcal{K}}}

\newcommand\Se{{\mathcal{S}}}

\newcommand\ind{{\mathbb{I}}}
\newcommand\inde{{\ind_0}}

\usepackage{graphics, graphicx, subfig}
\graphicspath{{figures/}}

\usepackage{tikz}
\usetikzlibrary{calc, shapes}

\tikzset{
	item/.style={minimum size=.5cm},
	customer/.style={item, fill=green!50},
	server/.style={item, fill=orange!50},
	fcfs/.style={
		draw,
		rectangle split,
		rectangle split parts=#1,
		rectangle split horizontal,
		rectangle split empty part width=.36cm,
		rectangle split empty part height=.51cm,
		inner xsep=0cm,
		inner ysep=0cm,
	},
}

\usepackage{pgfplots,pgfplotstable}
\pgfplotsset{compat=1.8}
\usepgfplotslibrary{fillbetween}

\pgfplotsset{
	table/col sep = {comma},
	table/search path = {./, Numerical-results/},
	defaultplot/.style={
		xlabel near ticks, ylabel near ticks,
		xtick style={draw=none}, ytick style={draw=none},
		every axis plot/.append style={
			thick,
		},
	},
}

\pgfplotsset{
	loadplotstyle/.style={defaultplot,
		xmin=0, ymin=0,
		grid=major,
		legend pos=north west,
		legend style={
			cells={anchor=west, align=left},
			at={(0, 1.25)},
			/tikz/every even column/.append style={column sep=0.2cm},
		},
		legend columns=6,
		legend cell align={left},
		width=.5\linewidth, height=.3\linewidth,
	},
	transitionplotstyle/.style={loadplotstyle,
		ylabel={Transition probability},
		xmin=0, xmax=1,
		ymin=0, ymax=0.3,
		legend style={
			cells={anchor=west, align=left},
			at={(1.1, 0.7)},
			/tikz/every even column/.append style={column sep=0.2cm},
		},
		legend columns=3,
	},
	probaplotstyle/.style={loadplotstyle,
		ylabel={Waiting probability},
		xmin=0, xmax=1,
		ymin=0, ymax=1,
	},
	meanplotstyle/.style={loadplotstyle,
		ylabel={Mean waiting time},
		xmin=0, xmax=1,
		ymin=0, ymax=20,
	},
}

\usepackage{filecontents}

\begin{document}
	\title{Performance Evaluation of Stochastic~Bipartite~Matching~Models\footnote{The final authenticated version is
	available online at
	\url{https://doi.org/10.1007/978-3-030-91825-5_26}.}}
	
	\author[1]{C\'eline Comte\footnote{Corresponding author.}}
	\author[2]{Jan-Pieter Dorsman}
	
	\affil[1]{Eindhoven University of Technology, The Netherlands,
		\href{mailto:c.m.comte@tue.nl}{c.m.comte@tue.nl}}
	\affil[2]{University of Amsterdam, The Netherlands,
		\href{mailto:j.l.dorsman@uva.nl}{j.l.dorsman@uva.nl}}
	\date{\vspace{-1cm}}
	
	\maketitle
	\begin{abstract}
		We consider a stochastic bipartite matching model consisting of multi-class customers and multi-class servers. Compatibility constraints between the customer and server classes are described by a bipartite graph. Each time slot, exactly one customer and one server arrive. The incoming customer (resp.\ server) is matched with the earliest arrived server (resp.\ customer) with a class that is compatible with its own class, if there is any, in which case the matched customer-server couple immediately leaves the system; otherwise, the incoming customer (resp.\ server) waits in the system until it is matched. Contrary to classical queueing models, both customers and servers may have to wait, so that their roles are interchangeable. While (the process underlying) this model was already known to have a product-form stationary distribution, this paper derives a new compact and manageable expression for the normalization constant of this distribution, as well as for the waiting probability and mean waiting time of customers and servers. We also provide a numerical example and make some important observations.
		
	\end{abstract}

	\section{Introduction}
	
	Stochastic matching models typically consist of items of multiple classes that arrive at random instants to be matched with items of other classes. In the same spirit as classical (static) matching models,
	stochastic models encode compatibility constraints between items using a graph on the classes.
	This allows for the modeling of many matching applications that are stochastic in nature, such as organ transplants where not every patient is compatible with every donor organ.
	
	In the literature on stochastic matching, a rough distinction is made between \emph{bipartite} and \emph{non-bipartite} models. In a bipartite matching model, the graph that describes compatibility relations between item classes is bipartite. In this way, item classes can be divided into two groups called \emph{customers} and \emph{servers}, so that customers (resp.\ servers) cannot be matched with one another. This is the variant that we consider in this paper. It is discrete-time in nature and assumes that, every time unit, exactly one customer and one server arrive. The classes of incoming customers and servers are drawn independently from each other, and they are also independent and identically distributed across time units. Following~\cite{BGM13,CKW09}, we adopt the common first-come-first-matched policy, whereby an arriving customer (resp.\ server) is matched with the earliest arriving compatible server (resp.\ customer). A toy example is shown in \figurename~\ref{fig:toy}. This model is equivalent to the \textit{first-come-first-served infinite bipartite matching model}, studied in \cite{ABMW17,AKRW18,CKW09}, which can be used to describe the evolution of waiting lists in public-housing programs and adoption agencies for instance~\cite{CKW09}.
	
	\begin{figure}[t]
		\centering
		\subfloat[Compatibility graph
		\label{fig:graph}]{%
			\begin{tikzpicture}
				\def\width{.9cm}
				\def\height{.965cm}
				
				\node[customer] (1) {1};
				\node[customer] (2)
				at ($(1)+(\width,0)$) {2};
				\node[customer] (3)
				at ($(2)+(\width,0)$) {3};
				\node[customer] (4)
				at ($(3)+(\width,0)$) {4};
				
				\node[server] (A)
				at ($(1)+(-.5*\width,-\height)$) {A};
				\node[server] (B)
				at ($(1)!.5!(2)+(0,-\height)$) {B};
				\node[server] (C)
				at ($(2)!.5!(3)+(0,-\height)$) {C};
				\node[server] (D)
				at ($(3)+(.5*\width,-\height)$) {D};
				\node[server] (E)
				at ($(4)+(.5*\width,-\height)$) {E};
				
				\draw[-] (A) -- (1) -- (B) -- (2) -- (C) -- (3) -- (D) -- (4) -- (E);
				
				\node[xshift=-1.7cm, anchor=west]
				at (1 -| A) {Customers};
				\node[xshift=-1.7cm, anchor=west]
				at (A) {Servers};
				\phantom{
					\node[anchor=north] at (A.south) {\strut};
				}
			\end{tikzpicture}
		}
		\qquad
		\subfloat[First-come-first-matched policy
		\label{fig:state}]{%
			\begin{tikzpicture}
				\def\height{.7cm}
				
				% queue of customers
				\node[fcfs=6,
				rectangle split part fill
				= {green!50, green!50, green!50,
					white, white, white},
				] (customers) {};
				\fill[white] ([xshift=\pgflinewidth+1pt,
				yshift=-\pgflinewidth+1pt]customers.north east)
				rectangle ([xshift=-\pgflinewidth-4pt,
				yshift=\pgflinewidth-1pt]customers.south east);
				\fill[white] ([xshift=\pgflinewidth+1pt,
				yshift=-\pgflinewidth-.01pt]customers.north east)
				rectangle ([xshift=-\pgflinewidth-5pt,
				yshift=\pgflinewidth+.01pt]customers.south east);
				\node[anchor=south]
				at ($(customers.north)$)
				{\strut State $c = (1,2,1)$};
				
				% state of this queue
				\node at
				($(customers.one north)!.5!(customers.one south)$) {1};
				\node at
				($(customers.two north)!.5!(customers.two south)$) {2};
				\node at
				($(customers.three north)!.5!(customers.three south)$) {1};
				
				% queue of servers
				\node[fcfs=6,
				rectangle split part fill
				= {orange!50, orange!50, orange!50,
					white, white, white},
				] (servers) at ($(customers.south)-(0,\height)$) {};
				\fill[white] ([xshift=\pgflinewidth+1pt,
				yshift=-\pgflinewidth+1pt]servers.north east)
				rectangle ([xshift=-\pgflinewidth-4pt,
				yshift=\pgflinewidth-1pt]servers.south east);
				\fill[white] ([xshift=\pgflinewidth+1pt,
				yshift=-\pgflinewidth-.01pt]servers.north east)
				rectangle ([xshift=-\pgflinewidth-5pt,
				yshift=\pgflinewidth+.01pt]servers.south east);
				\node[anchor=north]
				at ($(servers.south)$)
				{\strut State $d = (D,D,D)$};
				
				% state of this queue
				\node at
				($(servers.one north)!.5!(servers.one south)$) {D};
				\node at
				($(servers.two north)!.5!(servers.two south)$) {D};
				\node at
				($(servers.three north)!.5!(servers.three south)$) {D};
				
				% arrival
				\node[customer] (newcustomer)
				at ($(customers.west)-(1.4cm,0)$) {2};
				\node[server] (newserver)
				at ($(servers.west)-(1.4cm,0)$) {C};
				
				\node[anchor=south]
				at (newcustomer.north) {\strut New customer};
				\node[anchor=north]
				at (newserver.south) {\strut New server};
				
				\draw[->] ($(newcustomer.west)-(.6cm,0)$)
				-- (newcustomer);
				\draw[->] ($(newserver.west)-(.6cm,0)$)
				-- (newserver);
				
				\draw[->]
				($(newserver.east)+(.05cm,0)$)
				-- ($(customers.two south)-(.1cm,.05cm)$);
				
				% cross
				\draw[-] ($(customers.two)+(.5cm,.5cm)$)
				-- ($(customers.two)$);
				\draw[-] ($(customers.two)+(.5cm,0)$)
				-- ($(customers.two)+(0,.5cm)$);
			\end{tikzpicture}
		}
		\caption{%
			A stochastic bipartite matching model
			with a set $\I = \{1, 2, 3, 4\}$
			of customer classes
			and a set $\K = \{A, B, C, D, E\}$
			of server classes.
		}
		\label{fig:toy}
	\end{figure}
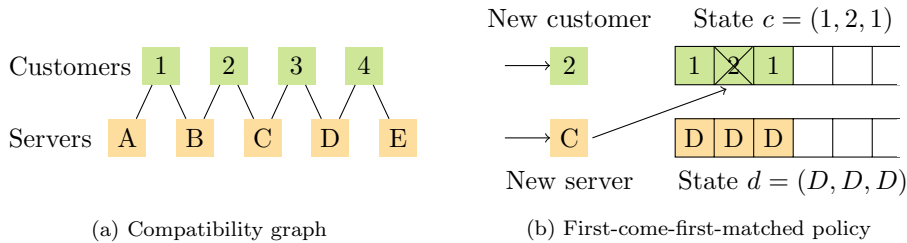
	
	In contrast, in a stochastic \emph{non}-bipartite matching model, item classes cannot be divided into two groups because the compatibility graph is non-bipartite. Another notable difference is that only one item arrives at each time slot, the classes of successive items being drawn independently from the same distribution. While \cite{MM16} derived stability conditions for this non-bipartite model, \cite{MBM21} showed that the stationary distribution of the process underlying this model has a product form. The recent work~\cite{C21} built on the latter result to derive closed-form expressions for several performance metrics, by also exploiting a connection with order-independent (loss) queues \cite{BK96,K11}. The present work seeks to provide a similar analysis for the above-mentioned bipartite model. 
	
	More specifically, we derive closed-form expressions for several performance metrics in the stochastic bipartite model studied in \cite{ABMW17}. While earlier studies~\cite{AW12,CKW09} were skeptical about the tractability of the stationary distribution corresponding to this variant, \cite{ABMW17} showed that this stationary distribution in fact possesses the product-form property, thus paving the way for an analysis similar to that of~\cite{C21}. That is, we use techniques from order-independent (loss) queues~\cite{BK96,K11} and other product-form models with compatibility constraints (cf.\ \cite{GR20} for a recent overview) to analyze the stochastic bipartite matching model. 
	
	The rest of this paper is structured as follows. In Section \ref{sec:model}, we introduce the model and cast it in terms of a framework commonly used for analysis of product-form models. We then provide a performance evaluation of this model. More specifically, we first derive an alternative closed-form expression for the normalization constant of the stationary distribution.
	While the computational complexity of this expression
	is prohibitive for instances with many classes
	(as was the case for the expression derived in~\cite{ABMW17}),
	it draws the relation with product-form queues and paves the way for heavy-traffic analysis. Furthermore, it allows us to directly derive recursive expressions for several other performance metrics, such as the probability that incoming customers or servers have to wait and the mean number of customers and servers that are waiting. To the best of the authors' knowledge, this paper is the first to provide expressions for these performance metrics. This analysis is presented in Section \ref{sec:performance}. Finally, Section~\ref{sec:numerical} numerically studies a model instance and makes some important observations.

	\section{Model and preliminary results} \label{sec:model}
	
	In Section~\ref{subsec:model}, we describe
	a stochastic bipartite matching model
	in which items of two groups,
	called customers and servers,
	arrive randomly
	and are matched with one another.
	As mentioned earlier, this model is analogous
	to that introduced in~\cite{CKW09}
	and further studied in~\cite{ABMW17,AKRW18,AW12,BGM13}.
	Section~\ref{subsec:dtmc}
	focuses on a discrete-time Markov chain
	that describes the evolution of this model.
	Finally, Section~\ref{subsec:steady} recalls
	several results that are useful for the analysis
		of Section~\ref{sec:performance}.
	
	\subsection{Model and notation} \label{subsec:model}
	
	\paragraph{Bipartite compatibility graph}
	
	Consider a finite set $\I$
	of $I$ customer classes
	and a finite set $\K$
	of $K$ server classes.
	Also consider
	a connected bipartite graph
	on the sets $\I$ and $\K$.
	For each $i \in \I$ and $k \in \K$,
	we write $i \sim k$
	if there is an edge between
	nodes~$i$ and~$k$ in this graph,
	and $i \nsim k$ otherwise.
	This bipartite graph is
	called the \emph{compatibility} graph of the model.
	It describes the compatibility relations
	between customers and servers
	in the sense that,
	for each $i \in \I$ and $k \in \K$,
	a class-$i$ customer
	and a class-$k$ server
	can be matched with one another
	if and only if
	there is an edge between
	the corresponding nodes
	in the compatibility graph.
	An example is shown in
	\figurename~\ref{fig:graph}.
	To simplify reading, we consistently use
	the letters $i, j \in \I$
	for customer classes
	and $k, \ell \in \K$
	for server classes.
	
	\paragraph{Discrete-time stochastic matching}
	
	Unmatched customers and servers
	are stored in two separate queues
	in their arrival order.
	In \figurename~\ref{fig:state},
	items are ordered
	from the oldest on the left
	to the newest on the right,
	and each item is labeled by its class.
	The two queues are initially empty.
	Time is slotted and,
	during each time slot,
	exactly one customer \emph{and}
	exactly one server arrive.
	The incoming customer
	belongs to class~$i$
	with probability $\lambda_i > 0$,
	for each $i \in \I$,
	and the incoming server
	belongs to class~$k$
	with probability $\mu_k > 0$
	for each $k \in \K$,
	with $\sum_{i \in \I} \lambda_i
	= \sum_{k \in \K} \mu_k = 1$.
	The classes of incoming customers and servers
	are independent within and across time slots.
	The matching policy,
	called \emph{first-come-first-matched},
	consists of applying
	the following four steps
	upon each arrival:
	\begin{enumerate}
		\item Match the incoming customer
		with the compatible unmatched server
		that has been in the queue the longest, if any.
		\item Match the incoming server
		with the compatible unmatched customer
		that has been in the queue the longest, if any.
		\item If neither the incoming customer
		nor the incoming server
		can be matched with unmatched items,
		match them together if they are compatible.
		\item If an incoming customer and/or incoming server remains unmatched after the previous steps, it is appended to the back of its respective queue.
	\end{enumerate}
	When two items are matched with one another,
	they immediately disappear.
	In the example of \figurename~\ref{fig:state},
	the couple (2, C) arrives
	while the sequence of unmatched
	customer and server classes are
	$(1, 2, 1)$ and $(D, D, D)$.
	According to the compatibility graph
	of \figurename~\ref{fig:graph},
	class~C is compatible
	with class~2 but not with class~1.
	Therefore, the incoming class-C server
	is matched with the second oldest unmatched customer,
	of class~2.
	The incoming class-2 customer
	is not matched with any present item
	(even if it is compatible
	with the incoming class-C server),
	therefore it is appended to
	the queue of unmatched customers.
	After this transition,
	the sequence of unmatched
	customer classes becomes $(1, 1, 2)$,
	while the sequence of unmatched
	server classes is unchanged.
	
	\begin{remark}
		If we would consider the random sequences
		of classes of incoming customers and servers,
		we would retrieve the state descriptor
		of the infinite bipartite matching model
		introduced in~\cite{CKW09}
		and studied in~\cite{ABMW17,AKRW18,AW12,BGM13}.
		For analysis purposes, we however adopted the above-introduced state descriptor consisting of the sequences
		of (waiting) unmatched customers and servers,
		corresponding to the
		\emph{natural pair-by-pair FCFS Markov chain}
		introduced in \cite[Section~2]{ABMW17}.
	\end{remark}
	
	\paragraph{Set notation}
	
	The following notation
	will be useful.
	Given two sets $\A$ and~$\B$,
	we write $\A \subseteq \B$
	if $\A$ is a subset of $\B$
	and $\A \subsetneq \B$
	if $\A$ is a proper subset of~$\B$.
	For each $i \in \I$,
	we let $\K_i \subseteq \K$ denote
	the set of server classes
	that can be matched
	with class-$i$ customers.
	Similarly, for each $k \in \K$,
	we let $\I_k \subseteq \I$ denote
	the set of customer classes
	that can be matched
	with class-$k$ servers.
	For each $i \in \I$ and $k \in \K$,
	the statements $i \sim k$,
	$i \in \I_k$, and $k \in \K_i$
	are equivalent.
	In \figurename~\ref{fig:graph} for instance,
	we have $\K_1 = \{A, B\}$, $\K_2 = \{B, C\}$,
		$\K_3 = \{C, D\}$, $\K_4 = \{D, E\}$
		$\I_A = \{1\}$, $\I_B = \{1, 2\}$,
		$\I_C = \{2, 3\}$, $\I_D = \{3, 4\}$,
		and $\I_E = \{4\}$.
	With a slight abuse of notation,
	for each $\A \subseteq \I$, we let
	$\lambda(\A) = \sum_{i \in \A} \lambda_i$
	denote the probability that
	the class of an incoming customer belongs to $\A$
	and $\K(\A) = \bigcup_{i \in \A} \K_i$
	the set of server classes
	that are compatible with customer classes in $\A$.
	Similarly, for each $\A \subseteq \K$,
	we write
	$\mu(\A) = \sum_{k \in \A} \mu_k$
	and $\I(\A) = \bigcup_{k \in \A} \I_k$.
	In particular, we have
	$\lambda(\I) = \mu(\K) = 1$,
	$\K(\I) = \K$, and $\I(\K) = \I$.

	\subsection{Discrete-time Markov chain} \label{subsec:dtmc}
	
	We now consider a Markov chain
	that describes the evolution of the system.
	
	\paragraph{System state}
	
	We consider the couple $(c, d)$, where
	$c = (c_1, \ldots, c_n) \in \I^*$
	is the sequence of unmatched customer classes,
	ordered by arrival,
	and $d = (d_1, \ldots, d_n) \in \K^*$
	is the sequence of unmatched server classes,
	ordered by arrival.
	In particular, $c_1$ is the class
	of the oldest unmatched customer, if any,
	and $d_1$ is the class
	of the oldest unmatched server, if any.
	The notation $\I^*$ (resp.\ $\K^*$) refers to
	the Kleene star on $\I$ (resp.\ $\K$), that is,
	the set of sequences
	of elements in $\I$ (resp.\ $\K$)
	with a length that is finite but arbitrarily large
	\cite[Chapter~1, Section~2]{D09}.
	As we will see later,
	the matching policy guarantees that the numbers
	of unmatched customers and servers
	are always equal to each other,
	and consequently the integer~$n$
	will be called the \emph{length} of the state.
	The empty state, with $n = 0$,
	is denoted by $\varnothing$.
	
	The evolution of this state over time
	defines a (discrete-time) Markov chain
	that is further detailed below.
	For each sequence $c = (c_1, \ldots, c_n) \in \I^*$,
	we let $|c| = n$ denote the length of sequence~$c$,
	$|c|_i$ the number of occurrences
	of class~$i$ in sequence~$c$, for each $i \in \I$,
	and, with a slight abuse of notation,
	$\{c_1, \ldots, c_n\}$ the set of classes
	that appear in sequence~$c$
	(irrespective of their multiplicity).
	Analogous notation is introduced
	for each sequence $d = (d_1, \ldots, d_n) \in \K^*$.
	
	\paragraph{Transitions}
	
	Each transition of the Markov chain
	is triggered by the arrival
	of a customer-server couple.
	We distinguish five types of transitions
	depending on their impact on the queues
	of unmatched customers and servers:
	\begin{itemize}[leftmargin=9.5mm]
		\item[\namedlabel{minus/minus}{$-$/$-$}]
		The incoming customer
		is matched with an unmatched server
		and the incoming server
		is matched with an unmatched customer.
		\item[\namedlabel{pm/equal}{$\pm$/$=$}]
		The incoming customer cannot
		be matched with any present server
		but the incoming server
		is matched with an unmatched customer.
		\item[\namedlabel{equal/pm}{$=$/$\pm$}]
		The incoming customer
		is matched with a present server
		but the incoming server cannot
		be matched with any present customer.
		\item[\namedlabel{equal/equal}{$=$/$=$}]
		Neither the incoming customer
		nor the incoming server
		can be matched with an unmatched item,
		but they are matched with one another.
		\item[\namedlabel{plus/plus}{$+$/$+$}]
		Neither the incoming customer
		nor the incoming server
		can be matched with an unmatched item,
		and they cannot be matched with one another.
	\end{itemize}
	Labels indicate the impact
	of the corresponding transition.
	For instance, a transition \ref{minus/minus}
	leads to a deletion ($-$) in the customer queue
	and a deletion ($-$) in the server queue,
	while a transition \ref{pm/equal}
	leads to a replacement ($\pm$) in the customer queue
	and no modification in the server queue ($=$).
	Transitions \ref{minus/minus}
	reduce the lengths of both queues by one,
	transitions \ref{pm/equal}, \ref{equal/pm},
	and \ref{equal/equal} leave the queue lengths unchanged,
	and transitions \ref{plus/plus}
	increase the lengths of both queues by one.
	Note that the numbers of unmatched customers and servers
	are always equal to each other.
	We omit the transition probabilities,
	as we will rely on an existing result
	giving the stationary distribution
	of the Markov chain.
	
	\paragraph{State space}
	
	The greediness of the matching policy
	prevents the queues from containing
	an unmatched customer and an unmatched server
	that are compatible.
	Therefore, the state space
	of the Markov chain
	is the subset of $\I^* \times \K^*$ given by
	\begin{align*}
		\Pi = \bigcup_{n = 0}^\infty
		\left\{
		(c, d) \in \I^n \times \K^n:
		c_p \nsim d_q
		\text{ for each }
		p, q \in \{1, \ldots, n\}
		\right\}.
	\end{align*}
	The Markov chain is irreducible.
	Indeed, using the facts that
	the compatibility graph is connected,
	that $\lambda_i > 0$ for each $i \in \I$,
	and that $\mu_k > 0$ for each $k \in \K$,
	we can show that
	the Markov chain can go
	from any state $(c, d) \in \Pi$
	to any state $(c', d') \in \Pi$
	via state $\varnothing$
	in $|c| + |c'| = |d| + |d'|$ jumps.
	
	\begin{remark} \label{remark:ctmc}
		We can also consider
		the following continuous-time variant
		of the model introduced
		in Section~\ref{subsec:model}.
		Instead of assuming that time is slotted,
		we can assume that customer-server couples
		arrive according to
		a Poisson process with unit rate.
		If the class of the incoming customers and servers
		are drawn independently at random,
		according to the probabilities
		$\lambda_i$ for $i \in \I$
		and $\mu_k$ for $k \in \K$,
		then the rate diagram
		of the continuous-time Markov chain
		describing the evolution
		of the sequences of unmatched items
		is identical to the transition diagram
		of the Markov chain
		introduced above.
		Consequently,
		the results recalled
		in Section~\ref{subsec:steady}
		and those derived in Section~\ref{sec:performance}
		can be applied without any modification
		to this continuous-time Markov chain.
	\end{remark}

	\subsection{Stability conditions and stationary distribution} \label{subsec:steady}
	
	For purposes of later analysis,
	we now state the following theorem,
	which was proved in
	\cite[Theorem~3]{AW12} and
	\cite[Lemma~2 and Theorems~2 and~8]{ABMW17}.
	
	\begin{theorem} \label{theo:distribution}
		The stationary measures of
		the Markov chain
		associated with the system state
		are of the form
		\begin{align} \label{eq:picd}
			\pi(c, d)
			= \pi(\varnothing)
			\prod_{p = 1}^n
			\frac{\lambda_{c_p}}
			{\mu(\K(\{c_1, \ldots, c_p\}))}
			\frac{\mu_{d_p}}
			{\lambda(\I(\{d_1, \ldots, d_p\}))},
			\quad (c, d) \in \Pi.
		\end{align}
		The system is stable,
		in the sense that this Markov chain is ergodic,
		if and only if one of the following
		two equivalent conditions is satisfied:
		\begin{align}
			\label{eq:stability-customers}
			&\lambda(\A) < \mu(\K(\A))
			\text{ for each non-empty set
				$\A \subsetneq \I$}, \\
			\label{eq:stability-servers}
			&\mu(\A) < \lambda(\I(\A))
			\phantom{,}
			\text{ for each non-empty set
				$\A \subsetneq \K$}.
		\end{align}
		In this case,
		the stationary distribution
		of the Markov chain
		associated with the system state
		is given by~\eqref{eq:picd},
		with the normalization constant
		\begin{align} \label{eq:normalization}
			\pi(\varnothing)
			&= \left(
			\sum_{(c, d) \in \Pi}
			\prod_{p = 1}^n
			\frac{\lambda_{c_p}}
			{\mu(\K(\{c_1, \ldots, c_p\}))}
			\frac{\mu_{d_p}}
			{\lambda(\I(\{d_1, \ldots, d_p\}))}
			\right)^{-1}.
		\end{align}
	\end{theorem}
	
	The states of the two queues
	are not independent in general
	because their lengths are equal.
	However, \eqref{eq:picd} shows that
	these two queue states
	are \emph{conditionally} independent
	\emph{given} the number~$n$ of unmatched items.
	This property will contribute
	to simplify the analysis
	in Section~\ref{sec:performance}.
	
	\begin{remark}
		The stationary measures~\eqref{eq:picd}
		seem identical to
		the stationary measures
		associated with another queueing model,
		called an FCFS-ALIS parallel queueing model
		\cite{AKRW18,AW14}.
		The only (crucial) difference
		lies in the definition
		of the state space
		of the corresponding Markov chain.
		In particular,
		our model imposes that the lengths
		of the two queues are equal to each other.
		In contrast,
		in the FCFS-ALIS parallel queueing model,
		there is an upper bound on
		the number of unmatched servers,
		while the number of customers
		can be arbitrarily large.
		This difference significantly
		changes the analysis.
		The analysis that
		we propose in Section~\ref{sec:performance}
		is based on the resemblance
		with another queueing model,
		called a multi-server queue for simplicity,
		that was introduced in~\cite{C19,G15}.
	\end{remark}

	\section{Performance evaluation by state aggregation} \label{sec:performance}
	
	We now assume that the stability
	conditions~\eqref{eq:stability-customers}\textendash\eqref{eq:stability-servers} are satisfied,
	and we let $\pi$ denote
	the stationary distribution,
	recalled in Theorem~\ref{theo:distribution},
	of the Markov chain
	of Section~\ref{subsec:dtmc}.
	Sections~\ref{subsec:empty}
	to~\ref{subsec:waiting-time}
	provide closed-form expressions
	for several performance metrics,
	based on a method explained
	in Section~\ref{subsec:partition}.
	The time complexity to implement these formulas
	and the relation with
	related works~\cite{ABMW17,AW12}
	are discussed in
	Section~\ref{subsec:discussion}.
	The reader who is not interested
	in understanding the proofs can move directly
	to Section~\ref{subsec:empty}.
	
	\subsection{Partition of the state space} \label{subsec:partition}
	
	A naive application of~\eqref{eq:normalization}
	does not allow calculation of the normalization constant,
	nor any other long-run performance metric as a result,
	because the state-space $\Pi$ is infinite.
	To circumvent this, we define a partition
	of the state space.
	
	\paragraph{Partition of the state space~$\Pi$}
	Let $\ind$ denote
	the family of sets $\A \subseteq \I \cup \K$
	such that
	$\A$ is an independent set
	of the compatibility graph
	and the sets
	$\A \cap \I$ and $\A \cap \K$ are non-empty.
	Also let $\inde
	= \ind \cup \{\emptyset\}$.
	For each $\A \in \inde$,
	we let $\Pi_{\A}$ denote
	the set of couples $(c, d) \in \Pi$
	such that $\{c_1, \ldots, c_n\} = \A \cap \I$
	and $\{d_1, \ldots, d_n\} = \A \cap \K$;
	in other words, $\Pi_\A$ is the set of states
	such that the set of unmatched classes is $\A$.
	We can show that
	$\{\Pi_{\A}, \A \in \inde\}$
	forms a partition of $\Pi$,
	and in particular
	\begin{align*}
		\Pi = \bigcup_{\A \in \inde} \Pi_{\A}.
	\end{align*}
	The first cornerstone of the subsequent analysis
	is the observation that,
	for each $(c,d) \in \Pi_{\A}$, we have
	$\mu(\K(\{c_1, \ldots, c_n\}))
	= \mu(\K(\A \cap \I))$
	and $\lambda(\I(\{d_1, \ldots, d_n\}))
	= \lambda(\I(\A \cap \K))$.
	In anticipation of Section~\ref{subsec:empty},
	for each $\A \in \ind$,
	we let
	\begin{align} \label{eq:delta}
		\Delta(\A)
		= \mu(\K(\A \cap \I)) \lambda(\I(\A \cap \K))
		- \lambda(\A \cap \I) \mu(\A \cap \K).
	\end{align}
	One can verify that $\Delta(\A) > 0$ for each $\A \in \ind$
	if and only if the stability conditions~\eqref{eq:stability-customers}%
	--\eqref{eq:stability-servers} are satisfied.
	The product $\lambda(\A \cap \I) \mu(\A \cap \K)$
	is the probability that
	an incoming client-server couple
	has its classes in $\A$,
	while the product
	$\mu(\K(\A \cap \I)) \lambda(\I(\A \cap \K))$
	is the probability that
	an incoming client-server couple
	can be matched with
	clients and servers
	whose classes belong to~$\A$.
	By analogy with
	the queueing models in~\cite{C21,G15},
	the former product can be seen
	as the ``arrival rate''
	of the classes in~$\A$,
	while the latter product can be seen
	as the maximal ``departure rate''
	of these classes.
	
	\paragraph{Partition of the subsets~$\Pi_\A$}
	
	The second cornerstone of the analysis
	is a partition of the set $\Pi_\A$ for each $\A \in \ind$.
	More specifically,
	for each $\A \in \ind$, we have
	\begin{align} \label{eq:recurrence}
		\Pi_\A
		= \bigcup_{i \in \A \cap \I}
		\bigcup_{k \in \A \cap \K}
		\left(
		\Pi_\A
		\cup \Pi_{\A {\setminus} \{i\}}
		\cup \Pi_{\A {\setminus} \{k\}}
		\cup \Pi_{\A {\setminus} \{i, k\}}
		\right) \cdot (i, k),
	\end{align}
	where $\Se \cdot (i, k)
	= \{((c_1, \ldots, c_n, i), (d_1, \ldots, d_n, k)):
	((c_1, \ldots, c_n), (d_1, \ldots, d_n)) \in \Se\}$
	for each $\Se \subseteq \Pi$,
	$i \in \I$, and $k \in \K$,
	and the unions are disjoint.
	Indeed, for each $(c, d) \in \Pi_\A$,
	the sequence $c = (c_1, \ldots, c_n)$
	can be divided into
	a prefix $(c_1, \ldots, c_{n-1})$
	and a suffix $i = c_n$;
	the suffix can take any value in $\A \cap \I$,
	while the prefix satisfies
	$\{c_1, \ldots, c_{n-1}\} = \A \cup \I$
	or $\{c_1, \ldots, c_{n-1}\}
	= (\A {\setminus} \{i\}) \cup \I$.
	Similarly, for each $(c, d) \in \Pi_\A$,
	the sequence
	$d = (d_1, \ldots, d_n)$
	can be divided into
	a prefix $(d_1, \ldots, d_{n-1})$
	and a prefix $k = d_n$;
	the prefix can take any value in $\A \cap \K$,
	while the prefix satisfies
	$\{d_1, \ldots, d_{n-1}\} = \A \cap \K$
	or $\{d_1, \ldots, d_{n-1}\}
	= (\A {\setminus} \{k\}) \cap \K$.

	\subsection{Normalization constant} \label{subsec:empty}
	
	The first performance metric that we consider
	is the probability that the system is empty.
	According to~\eqref{eq:normalization},
	this is also the inverse
	of the normalization constant.
	With a slight abuse of notation, we first let
	\begin{align*}
		\pi(\A) = \sum_{(c,d) \in \Pi_{\A}} \pi(c, d),
		\quad \A \in \inde.
	\end{align*}
	To simplify notation,
	we adopt the convention that
	$\pi(\A) = 0$ if $\A \notin \inde$.
	The following proposition,
	combined with the normalization equation
	$\sum_{\A \in \inde} \pi(\A) = 1$,
	allows us to calculate
	the probability $\pi(\emptyset) = \pi(\varnothing)$
	that the system is empty.
	
	\begin{proposition} \label{prop:piAB}
		The stationary distribution of
		the set of unmatched item classes
		satisfies the recursion
		\begin{align}
			\nonumber
			\Delta(\A) \pi(\A)
			={}&\mu(\A \cap \K) \sum_{i \in \A \cap \I}
			\lambda_i \pi(\A {\setminus} \{i\})
			+ \lambda(\A \cap \I) \sum_{k \in \A \cap \K}
			\mu_k \pi(\A {\setminus} \{k\}) \\
			\label{eq:piAB}
			&+ \sum_{i \in \A \cap \I}
			\sum_{k \in \A \cap \K}
			\lambda_i \mu_k
			\pi(\A {\setminus} \{i, k\}),
			\quad \A \in \ind.
		\end{align}
	\end{proposition}
	
	\begin{proof}
		Let $\A \in \ind$.
		Substituting~\eqref{eq:picd}
		into the definition of $\pi(\A)$ yields
		\begin{align*}
			\pi(\A)
			&= \sum_{(c, d) \in \Pi_\A}
			\prod_{p = 1}^n
			\frac{\lambda_{c_p}}
			{\mu(\K(\{c_1, \ldots, c_p\}))}
			\frac{\mu_{d_p}}
			{\lambda(\I(\{d_1, \ldots, d_p\}))}, \\
			&= \sum_{(c, d) \in \Pi_\A}
			\frac{\lambda_{c_n}}{\mu(\K(\A \cap \I))}
			\frac{\mu_{d_n}}{\lambda(\I(\A \cap \K))}
			\pi((c_1, \ldots, c_{n-1}), (d_1, \ldots, d_{n-1})).
		\end{align*}
		Then, by applying~\eqref{eq:recurrence}
		and making a change of variable,
		we obtain
		\begin{align}
			\mu(\K(\A \cap \I))
			\lambda(\I(\A \cap \K))
			\pi(\A)
			= 
			\label{eq:piAB-proof}
			\sum_{i \in \A \cap \I}
			\sum_{k \in \A \cap \K}
			\lambda_i \mu_k
			\big(
			\pi(\A)
			+ \pi(\A {\setminus} \{i\})
			+ \pi(\A {\setminus} \{k\})
			+ \pi(\A {\setminus} \{i, k\})
			\big).
		\end{align}
		The result follows by rearranging the terms.
	\end{proof}
	
	\subsection{Waiting probability} \label{subsec:waiting-probability}
	
	The second performance metric that we consider is
	the waiting probability,
	that is, the probability that
	an item cannot be matched
	with another item upon arrival.
	The waiting probabilities
	of the customers and servers
	of each class can again be calculated
	using Proposition~\ref{prop:piAB},
	as they are given by
	\begin{align*}
		\omega_i
		&= \sum_{\substack{\A \in \inde:
				\A \cap \K_i = \emptyset}}
		\left(
		1 - \sum_{\substack{k \in \K_i
				{\setminus} \K(\A \cap \I)}} \mu_k
		\right)
		\pi(\A),
		\quad i \in \I, \\
		\omega_k
		&= \sum_{\substack{\A \in \inde:
				\A \cap \I_k = \emptyset}}
		\left(
		1 - \sum_{\substack{i \in \I_k
				{\setminus} \I(\A \cap \K)}} \lambda_i
		\right)
		\pi(\A),
		\quad k \in \K.
	\end{align*}
	If we consider the continuous-time variant
	described in Remark~\ref{remark:ctmc},
	these equations follow directly
	from the PASTA property.
	That this result also holds
	for the discrete-time variant of the model
	follows from the fact that
	the transition diagrams
	and stationary distributions
	of both models are identical.
	
	Corollary~\ref{coro:busy-sequence} below
	follows from Proposition~\ref{prop:piAB}.
	It shows that
	the probability that
	both the incoming customer
	and the incoming server
	can be matched with present items
	(corresponding to transitions \ref{minus/minus})
	is equal to the probability that
	both the incoming customer
	and the incoming server
	have to wait
	(corresponding to transitions \ref{plus/plus}).
	The proof is given in the appendix.
	
	\begin{corollary} \label{coro:busy-sequence}
		The following equality is satisfied:
		\begin{align} \label{eq:busy-sequence}
			\sum_{(i, k) \in \I \times \K}
			\lambda_i \mu_k
			\sum_{\substack{
					\A \in \ind:
					i \in \I(\A \cap \K), \\
					k \in \K(\A \cap \I)
			}}
			\pi(\A)
			&=
			\sum_{\substack{
					(i, k) \in \I \times \K: \\
					i \nsim k
			}}
			\lambda_i \mu_k
			\sum_{\substack{
					\A \in \inde:
					i \notin \I(\A \cap \K), \\
					k \notin \K(\A \cap \I)
			}} \pi(\A).
		\end{align}
	\end{corollary}
	
	This corollary means that, in the long run,
	the rate at which the queue lengths increase
	is equal to
	the rate at which the queue lengths decrease.
	Equation~\eqref{eq:busy-sequence}
	is therefore satisfied by every matching policy
	that makes the system stable. This equation
	also has the following graphical interpretation.
	Consider a \emph{busy sequence} of the system,
	consisting of a sequence of customer classes
	and a sequence of server classes
	that arrive between
	two consecutive instants
	when both queues are empty.
	We construct a bipartite graph,
	whose nodes are the elements
	of these two sequences,
	by adding an edge between customers and servers
	that arrive at the same time
	or are matched with one another.
	An example is shown
	in~\figurename~\ref{fig:busy-sequence}
	for the compatibility graph
	of \figurename~\ref{fig:graph}.
	If we ignore the customer-server couples
	that arrive at the same time
	and are also matched with one another,
	we obtain a 2-regular graph,
	that is, a graph where all nodes have degree two.
	Such a graph consists of one or more disconnected cycles.
	We define a left (resp.\ right) extremity
		as a vertical edge adjacent
		only to edges moving to the right (resp.\ left);
		such an edge represents
		a \ref{plus/plus}
		(resp.\ \ref{minus/minus})
		transition.
		One can verify that each cycle
		contains as many
		left extremities as right extremities.
	In the example of
	\figurename~\ref{fig:busy-sequence},
	after eliminating the couple 1--A,
	we obtain two disconnected cycles.
	The cycle depicted with a solid line
	has one left extremity (2--A)
	and one right extremity (4--C).
	The cycle depicted with a dashed line
	also has one left extremity (3--E)
	and one right extremity (4--C).
	Since stability means that
	the mean length of a busy sequence is finite,
	combining this observation
	with classical results from renewal theory
	gives an alternative proof
	that~\eqref{eq:busy-sequence}
	is satisfied by every matching policy
	that makes the system stable.
	
	\begin{figure}[ht]
		\centering
		\begin{tikzpicture}
			\def\width{1.1cm}
			\def\height{1.1cm}
			
			\node[customer] (1) {2};
			\node[customer] (2)
			at ($(1)+(\width,0)$) {3};
			\node[customer] (3)
			at ($(2)+(\width,0)$) {4};
			\node[customer] (4)
			at ($(3)+(\width,0)$) {1};
			\node[customer] (5)
			at ($(4)+(\width,0)$) {2};
			\node[customer] (6)
			at ($(5)+(\width,0)$) {1};
			\node[customer] (7)
			at ($(6)+(\width,0)$) {4};
			
			\node[server] (A)
			at ($(1)+(0,-\height)$) {A};
			\node[server] (B)
			at ($(2)+(0,-\height)$) {E};
			\node[server] (C)
			at ($(3)+(0,-\height)$) {D};
			\node[server] (D)
			at ($(4)+(0,-\height)$) {D};
			\node[server] (E)
			at ($(5)+(0,-\height)$) {B};
			\node[server] (F)
			at ($(6)+(0,-\height)$) {A};
			\node[server] (G)
			at ($(7)+(0,-\height)$) {C};
			
			\draw[-, blue]
			(1.south) -- (E.north) -- (5.south)
			-- (G.north) -- (7.south)
			-- (D.north) -- (4.south) -- (A.north) 
			-- (1.south);
			\draw[-, red, densely dashed]
			(2.south) -- (C.north) -- (3.south)
			-- (B.north) -- (2.south);
			\draw[-, violet, densely dashdotted]
			(6.south) -- (F.north);
			
			% place holders
			\phantom{
				\node at ($(1)-(2cm,0)$) {};
				\node at ($(7)+(2cm,0)$) {};
			}
		\end{tikzpicture}
		\caption{A busy sequence
			associated with the compatibility graph of
			\figurename~\ref{fig:graph}.
			The arrival sequences
			are 2, 3, 4, 1, 2, 1, 4
			and A, E, D, D, B, A, C.
			Each component of
			the corresponding bipartite graph
			is depicted with a different line style.%
			\label{fig:busy-sequence-graph}}
		\label{fig:busy-sequence}
	\end{figure}
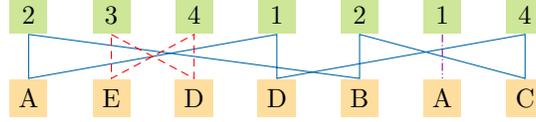
	
	\subsection{Mean number of unmatched items and mean waiting time} \label{subsec:waiting-time}
	
	We now turn to
	the mean number of unmatched items.
	Proposition~\ref{prop:LipiAB}
	gives a closed-form expression
	for the mean number of unmatched
	customers of each class.
	Proposition~\ref{prop:LIpiAB}
	gives a simpler expression
	for the mean number of unmatched
	customers (all classes included).
The proofs are similar to that of Proposition~\ref{prop:piAB}, with a few technical complications, and are deferred to the appendix.
	Analogous results can be obtained
	for the servers by using
	the model symmetry.
	
	\begin{proposition} \label{prop:LipiAB}
		For each $i \in \I$,
		the mean number of unmatched
		class-$i$ customers
		is $L_i = \sum_{\A \in \inde}
		\ell_i(\A)$,
		where $\ell_i(\A) / \pi(\A)$
		is the mean number of unmatched
		class-$i$ customers
		given that the set of unmatched classes is $\A$,
		and satisfies the recursion
		\begin{align}
			\nonumber
			\Delta(\A) \ell_i(\A)
			={}
			&
			\lambda_i \mu(\A \cap \K) \left(
			\pi(\A) + \pi(\A {\setminus} \{i\})
			\right)
			+ \lambda_i
			\sum_{k \in \A \cap \K} \mu_k \left(
			\pi(\A {\setminus} \{k\})
			+ \pi(\A {\setminus} \{i, k\})
			\right) \\
			\nonumber
			&+ \mu(\A \cap \K)
			\sum_{j \in \A}
			\lambda_j \ell_i(\A {\setminus} \{j\})
			+\lambda(\A \cap \I) \sum_{k \in \A \cap \K}
			\mu_k \ell_i(\A {\setminus} \{k\}) \\
			\label{eq:LipiAB}
			&+ \sum_{j \in \A}
			\sum_{k \in \A \cap \K}
			\lambda_j \mu_k
			\ell_i(\A {\setminus} \{j, k\}),
		\end{align}
		for each $\A \in \ind$ such that $i \in \A$,
		with the base case
		$\ell_i(\A) = 0$ if $i \notin \A$
		and the convention that $\ell_i(\A) = 0$
		if $\A \notin \inde$.
	\end{proposition}
	
	\begin{proposition} \label{prop:LIpiAB}
		The mean number of
		unmatched customers
		is $L_\I = \sum_{\A \in \inde}
		\ell_\I(\A)$,
		where $\ell_\I(\A) / \pi(\A)$ 
		is the mean number of unmatched customers
			given that the set of unmatched classes is $\A$,
			and satisfies the recursion
		\begin{align}
			\nonumber
			\Delta(\A) \ell_\I(\A)
			={}
			&\mu(\K(\A \cap \I))
			\lambda(\I(\A \cap \K)) \pi(\A) \\
			\nonumber
			&+ \mu(\A \cap \K) \sum_{i \in \A \cap \I}
			\lambda_i \ell_\I(\A {\setminus} \{i\})
			+ \lambda(\A \cap \I)
			\sum_{k \in \A \cap \K}
			\mu_k \ell_\I(\A {\setminus} \{k\}) \\
			\label{eq:LIpiAB}
			&+ \sum_{i \in \A \cap \I} \sum_{k \in \A \cap \K}
			\lambda_i \mu_k
			\ell_\I(\A {\setminus} \{i, k\}).
		\end{align}
		for each
		$\A \in \ind$,
		with the base case
		$\ell_\I(\emptyset) = 0$
		and the convention that
		$\ell_\I(\A) = 0$
		for each $\A \notin \inde$.
	\end{proposition}
	
	By Little's law,
	the mean waiting time
	of class-$i$ customers is
	$L_i / \mu_i$,
	for each $i \in \I$,
	and the mean waiting time
	of customers (all classes included) is $L$.
	By following the same approach as
	\cite[Propositions~9 and~10]{C21},
	we can derive, for each class,
	closed-form expressions
	for the distribution transforms
	of the number of unmatched items
	and waiting time.
	In the interest of space,
	and to avoid complicated notation,
	these results are omitted.

	\subsection{Time complexity and related work} \label{subsec:discussion}
	
	To conclude Section~\ref{sec:performance},
	we briefly discuss the merit
	of our approach compared to
	the expression derived
	in~\cite[Theorem~3]{AW12}
	and rederived in~\cite[Theorem~7]{ABMW17}
	for the normalization constant
	(equal to the inverse of
	the probability that the system is empty).
	This approach relies on a Markov chain
	called the \emph{server-by-server FCFS
		augmented matching process}
	in \cite[Section~5.4]{ABMW17}.
	
	\paragraph{Flexibility}
	
	The first merit of our approach
	is that it can be almost straightforwardly applied
	to derive other relevant performance metrics.
	Sections~\ref{subsec:waiting-probability}
	and~\ref{subsec:waiting-time}
	provide two examples:
	the expression of the waiting probability
	is a side-result of
	Proposition~\ref{prop:piAB},
	while the mean waiting time
	follows by a derivation
	along the same lines.
	Performance metrics
	that can be calculated in a similar fashion include
	the variance of the stationary number
	of unmatched items of each class,
	the mean length of a busy sequence, and
	the fractions of transitions of types
	\ref{minus/minus},
	\ref{pm/equal},
	\ref{equal/pm},
	\ref{equal/equal}, and
	\ref{plus/plus}.
	Our approach may also be adapted
	to derive an alternative expression
	for the matching rates
	calculated in \cite[Section~3]{AW12}.
	Indeed, upon applying the PASTA property,
	it suffices to calculate the stationary distribution
	of the \emph{order} of first occurrence
	of unmatched classes in the queues
	(rather than just the \emph{set}
	of unmatched classes);
	this distribution can be evaluated
	by considering a refinement of the partition
	introduced in Section~\ref{subsec:partition}.
	
	\paragraph{Time complexity}
	
	Compared to the formula
	of~\cite[Theorem~3]{AW12},
	our method leads to a lower time complexity
	if the number of independent sets
	in the compatibility graph is smaller
	than the cardinalities of the power sets
	of the sets $\I$ and $\K$.
	This is the case, for instance, in $d$-regular graphs,
	where the number of independent sets
	is at most $(2^{d+1} - 1)^{(I+K)/2d}$~\cite{Z10}.
	To illustrate this, let us first recall
	how to compute the probability
	that the system is empty
	using Proposition~\ref{prop:piAB}.
	The idea is to first
	apply~\eqref{eq:piAB} recursively
	with the base case $\pi(\emptyset) = 1$,
	and then derive the value of $\pi(\emptyset)$
	by applying the normalization equation.
	For each $\A \in \ind$,
	assuming that the values of
	$\pi(\A {\setminus} \{i\})$,
	$\pi(\A {\setminus} \{k\})$, and
	$\pi(\A {\setminus} \{i,k\})$
	are known
	for each $i \in \A \cap \I$
	and $k \in \A \cap \K$,
	evaluating $\pi(\A)$
	using~\eqref{eq:piAB}
	requires $O(I \cdot K)$
	operations,
	where $I$ is the number of customer classes
	and $K$ is the number of server classes.
	The time complexity
	to evaluate the probability that the system is empty
	is therefore given by
	$O(T + N \cdot I \cdot K)$,
	where $N$ is
	the number of independent sets
	in the compatibility graph
	and $T$ is the time complexity
	to enumerate all maximal independent sets.
	The result of~\cite{T77} implies
	that the time complexity to enumerate
	all maximal independent sets
	in the (bipartite) compatibility graph
	$O((I + K) \cdot I \cdot K \cdot M)$,
	where $M$ is the number
	of maximal independent sets.
	Overall, the time complexity to evaluate
	the normalization constant
	using Proposition~\ref{prop:piAB} is
	$O(I \cdot K \cdot ((I + K) \cdot M + N))$.
	
	In comparison,
	the time complexity to evaluate
	the normalization constant
	using \cite[Theorem~3]{AW12}
	is $O((I + K) \cdot 2^{\min(I, K)})$
	if we implement these formulas recursively,
	in a similar way as in~\cite{BCM17}.
	Our method thus leads
	to a lower time complexity
	if the number of independent sets
	of the compatibility graph is small.

	\section{Numerical evaluation} \label{sec:numerical}
		
	\begin{figure}[b!]
		\centering
		\subfloat[Customer-oriented performance metrics \label{fig:customers}]{
			\begin{tabular}[b]{c}
				\begin{tikzpicture}[every node/.style={scale=.8}]
					\begin{axis}[probaplotstyle, hide axis]
						\addlegendimage{teal, no markers}
						\addlegendentry{Average};
						
						\addlegendimage{orange, densely dashed, no markers}
						\addlegendentry{Class~1};
						
						\addlegendimage{green, densely dashdotted, no markers}
						\addlegendentry{Class~2};
						
						\addlegendimage{red, densely dotted, no markers,}
						\addlegendentry{Class~3};
						
						\addlegendimage{purple, dashed, no markers}
						\addlegendentry{Class~4};
					\end{axis}
				\end{tikzpicture}
				\\
				\pgfplotstableread{Results_model_waiting_probability_customers.csv}\model
				\pgfplotstableread{Results_simulation_waiting_probability_customers.csv}\simulation
				\begin{tikzpicture}[every node/.style={scale=.8}]
					\begin{axis}[probaplotstyle]
						% average
						\addplot+[
						teal, no markers,
						] table[x=parameter, y=5]{\model};
						
						\addplot+[teal, only marks, mark=x]
						table[x=parameter, y=5]{\simulation};
						
						% class 1
						\addplot+[
						orange, densely dashed, no markers,
						] table[x=parameter, y=1]{\model};
						
						\addplot+[orange, only marks, mark=x]
						table[x=parameter, y=1]{\simulation};
						
						% class 2
						\addplot+[
						green, densely dashdotted, no markers,
						] table[x=parameter, y=2]{\model};
						
						\addplot+[green, only marks, mark=x]
						table[x=parameter, y=2]{\simulation};
						
						% class 3
						\addplot+[
						red, densely dotted, no markers,
						] table[x=parameter, y=3]{\model};
						
						\addplot+[red, only marks, mark=x]
						table[x=parameter, y=3]{\simulation};
						
						% class 4
						\addplot+[
						purple, dashed, no markers,
						] table[x=parameter, y=4]{\model};
						
						\addplot+[purple, only marks, mark=x]
						table[x=parameter, y=4]{\simulation};
					\end{axis}
				\end{tikzpicture}
				\pgfplotstableread{Results_model_waiting_time_customers.csv}\model
				\pgfplotstableread{Results_simulation_waiting_time_customers.csv}\simulation
				\begin{tikzpicture}[every node/.style={scale=.8}]
					\begin{axis}[meanplotstyle]
						% average
						\addplot+[
						teal, no markers,
						] table[x=parameter, y=5]{\model};
						
						\addplot+[teal, only marks, mark=x]
						table[x=parameter, y=5]{\simulation};
						
						% class 1
						\addplot+[
						orange, densely dashed, no markers,
						] table[x=parameter, y=1]{\model};
						
						\addplot+[orange, only marks, mark=x]
						table[x=parameter, y=1]{\simulation};
						
						% class 2
						\addplot+[
						green, densely dashdotted, no markers,
						] table[x=parameter, y=2]{\model};
						
						\addplot+[green, only marks, mark=x]
						table[x=parameter, y=2]{\simulation};
						
						% class 3
						\addplot+[
						red, densely dotted, no markers,
						] table[x=parameter, y=3]{\model};
						
						\addplot+[red, only marks, mark=x]
						table[x=parameter, y=3]{\simulation};
						
						% class 4
						\addplot+[
						purple, dashed, no markers,
						] table[x=parameter, y=4]{\model};
						
						\addplot+[purple, only marks, mark=x]
						table[x=parameter, y=4]{\simulation};
					\end{axis}
				\end{tikzpicture}
			\end{tabular}
		}
		\hfill
		\subfloat[Server-oriented performance metrics \label{fig:servers}]{
			\begin{tabular}[b]{c}
				\begin{tikzpicture}[every node/.style={scale=.8}]
					\begin{axis}[probaplotstyle, hide axis]
						\addlegendimage{teal, no markers}
						\addlegendentry{Average};
						
						\addlegendimage{orange, densely dashed, no markers}
						\addlegendentry{Class~A};
						
						\addlegendimage{green, densely dashdotted, no markers}
						\addlegendentry{Class~B};
						
						\addlegendimage{red, densely dotted, no markers,}
						\addlegendentry{Class~C};
						
						\addlegendimage{purple, dashed, no markers}
						\addlegendentry{Class~D};
						
						\addlegendimage{olive, dashdotted, no markers}
						\addlegendentry{Class~E};
					\end{axis}
				\end{tikzpicture}
				\\
				\pgfplotstableread{Results_model_waiting_probability_servers.csv}\model
				\pgfplotstableread{Results_simulation_waiting_probability_servers.csv}\simulation
				\begin{tikzpicture}[every node/.style={scale=.8}]
					\begin{axis}[probaplotstyle]
						% average
						\addplot+[
						teal, no markers,
						] table[x=parameter, y=6]{\model};
						
						\addplot+[teal, only marks, mark=x]
						table[x=parameter, y=6]{\simulation};
						
						% class 1
						\addplot+[
						orange, densely dashed, no markers,
						] table[x=parameter, y=1]{\model};
						
						\addplot+[orange, only marks, mark=x]
						table[x=parameter, y=1]{\simulation};
						
						% class 2
						\addplot+[
						green, densely dashdotted, no markers,
						] table[x=parameter, y=2]{\model};
						
						\addplot+[green, only marks, mark=x]
						table[x=parameter, y=2]{\simulation};
						
						% class 3
						\addplot+[
						red, densely dotted, no markers,
						] table[x=parameter, y=3]{\model};
						
						\addplot+[red, only marks, mark=x]
						table[x=parameter, y=3]{\simulation};
						
						% class 4
						\addplot+[
						purple, dashed, no markers,
						] table[x=parameter, y=4]{\model};
						
						\addplot+[purple, only marks, mark=x]
						table[x=parameter, y=4]{\simulation};
						
						% class 5
						\addplot+[
						olive, dashdotted, no markers,
						] table[x=parameter, y=5]{\model};
						
						\addplot+[olive, only marks, mark=x]
						table[x=parameter, y=5]{\simulation};
					\end{axis}
				\end{tikzpicture}
				\pgfplotstableread{Results_model_waiting_time_servers.csv}\model
				\pgfplotstableread{Results_simulation_waiting_time_servers.csv}\simulation
				\begin{tikzpicture}[every node/.style={scale=.8}]
					\begin{axis}[meanplotstyle]
						% average
						\addplot+[
						teal, no markers,
						] table[x=parameter, y=6]{\model};
						
						\addplot+[teal, only marks, mark=x]
						table[x=parameter, y=6]{\simulation};
						
						% class 1
						\addplot+[
						orange, densely dashed, no markers,
						] table[x=parameter, y=1]{\model};
						
						\addplot+[orange, only marks, mark=x]
						table[x=parameter, y=1]{\simulation};
						
						% class 2
						\addplot+[
						green, densely dashdotted, no markers,
						] table[x=parameter, y=2]{\model};
						
						\addplot+[green, only marks, mark=x]
						table[x=parameter, y=2]{\simulation};
						
						% class 3
						\addplot+[
						red, densely dotted, no markers,
						] table[x=parameter, y=3]{\model};
						
						\addplot+[red, only marks, mark=x]
						table[x=parameter, y=3]{\simulation};
						
						% class 4
						\addplot+[
						purple, dashed, no markers,
						] table[x=parameter, y=4]{\model};
						
						\addplot+[purple, only marks, mark=x]
						table[x=parameter, y=4]{\simulation};
						
						% class 5
						\addplot+[
						olive, dashdotted, no markers,
						] table[x=parameter, y=5]{\model};
						
						\addplot+[olive, only marks, mark=x]
						table[x=parameter, y=5]{\simulation};
					\end{axis}
				\end{tikzpicture}
			\end{tabular}
		}
		\\
		\subfloat[Transition probabilities \label{fig:transitions}]{
			\pgfplotstableread{Results_model_transition_probabilities.csv}\model
			\pgfplotstableread{Results_simulation_transition_probabilities.csv}\simulation
			\begin{tikzpicture}[every node/.style={scale=.8}]
				\begin{axis}[transitionplotstyle]
					% minus/minus
					\addplot+[
					teal, no markers,
					] table[x=parameter, y=minus/minus]{\model};
					\addlegendentry{\ref{minus/minus}}
					
					\addplot+[teal, only marks, mark=x, forget plot]
					table[x=parameter, y=minus/minus]{\simulation};
					
					% pm/equal
					\addplot+[
					orange, densely dashed, no markers,
					] table[x=parameter, y=pm/equal]{\model};
					\addlegendentry{\ref{pm/equal}}
					
					\addplot+[orange, only marks, mark=x, forget plot]
					table[x=parameter, y=pm/equal]{\simulation};
					
					% equal/pm
					\addplot+[
					green, densely dashdotted, no markers,
					] table[x=parameter, y=equal/pm]{\model};
					\addlegendentry{\ref{equal/pm}}
					
					\addplot+[green, only marks, mark=x, forget plot]
					table[x=parameter, y=equal/pm]{\simulation};
					
					% equal/equal
					\addplot+[
					red, densely dotted, no markers,
					] table[x=parameter, y=equal/equal]{\model};
					\addlegendentry{\ref{equal/equal}}
					
					\addplot+[red, only marks, mark=x, forget plot]
					table[x=parameter, y=equal/equal]{\simulation};
					
					% plus/plus
					\addplot+[
					purple, dashed, no markers,
					] table[x=parameter, y=plus/plus]{\model};
					\addlegendentry{\ref{plus/plus}}
					
					\addplot+[purple, only marks, mark=x, forget plot]
					table[x=parameter, y=plus/plus]{\simulation};
				\end{axis}
			\end{tikzpicture}
			\hspace{2.67cm}
		}
		\caption{Numerical results associated
			with the graph of \figurename~\ref{fig:graph}.
			The abscissa is the parameter~$\rho$
			defined in~\eqref{eq:arrival}.}
		\label{fig:num}
	\end{figure}
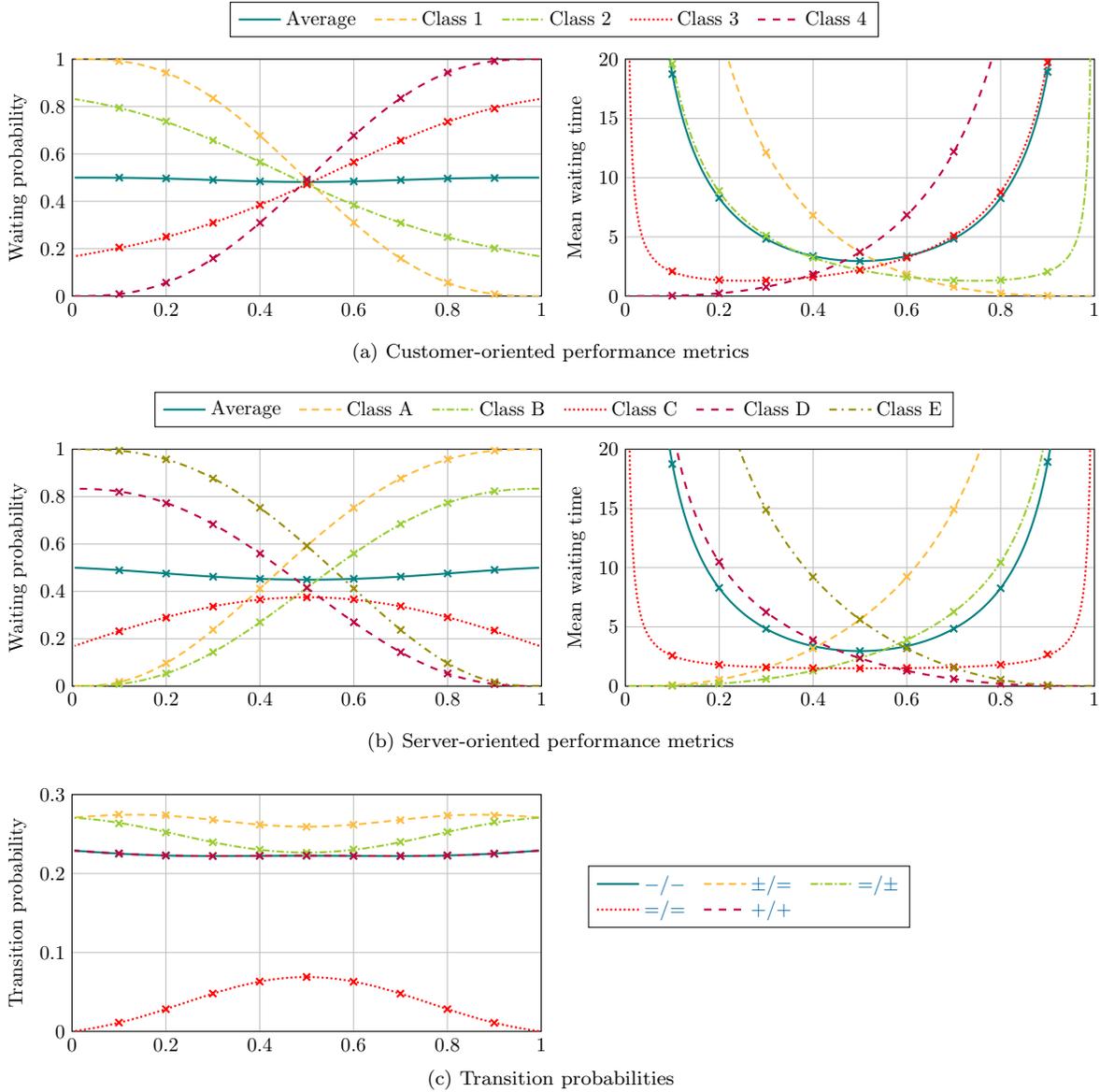

	To illustrate our results,
	we apply the formulas
	of Section~\ref{sec:performance}
	to the toy example
	of \figurename~\ref{fig:graph}.
	The arrival probabilities are chosen as follows:
	for any $\rho \in (0, 1)$,
	\begin{align} \label{eq:arrival}
	\lambda_1 &= \lambda_2 = \lambda_3 = \lambda_4 = \frac14, &
	\mu_A &= \frac\rho4, &
	\mu_B &= \mu_C = \mu_D = \frac14, &
	\mu_E &= \frac{1 - \rho}4.
	\end{align}
	\figurename~\ref{fig:num} shows	several performance metrics. The lines are plotted using the results of Section~\ref{sec:performance}. To verify these results, we plotted marks representing simulated values based on averaging the results of 20 discrete-event simulation runs, each consisting of $10^6$ transitions	after a warm-up period of $10^6$ transitions. The standard deviation of the simulated waiting times (resp.\ probabilities) never exceeded 1.9 (resp.\ 0.008) per experiment, validating the reliability of the results.
	
	Due to the parameter settings,
	performance is symmetrical around $\rho = \frac12$.
	\figurename~\ref{fig:customers}
	and~\ref{fig:servers} show that
	classes $1$, $2$, $3$, $C$, $D$, and $E$
	become unstable,
	in the sense that their mean waiting time
	tends to infinity,
	as $\rho \downarrow 0$.
	This is confirmed by observing that
	$\Delta(\A) \downarrow 0$
	for $\A \in \{ \{4, A\}$, $\{4, A, B\}$,
	$\{4, A, B, C\}$, $\{3, 4, A\}$,
	$\{3, 4, A, B\}$, $\{2, 3, 4, A\} \}$
	when $\rho \downarrow 0$.
	We conjecture that this limiting regime
	can be studied by adapting
	the heavy-traffic analysis
	of \cite[Section~6.2]{C21},
	although the behavior is different
	due to the concurrent arrivals
	of customers and servers.
	
	Even if classes $1$, $2$, $3$, $C$, $D$, and $E$
	all become unstable as $\rho \downarrow 0$,
	we can distinguish two
	qualitatively-different behaviors:
	the waiting probabilities of
	classes~$1$ and~$E$ tend to one,
	while for classes~$2$, $3$, $C$, and $D$
	the limit is strictly less than one.
	This difference lies in the fact that
	the former classes have degree one in the compatibility graph, while the latter have degree two. Especially class~$C$ is intriguing,
	as the monotonicity of its
	waiting probability and mean waiting time
	are reversed,
	and would be worth further investigation.
	
	\figurename~\ref{fig:transitions} shows that
	the probabilities of
	transitions~\ref{minus/minus}
	and~\ref{plus/plus}
	are equal to each other
	(as announced by Corollary~\ref{coro:busy-sequence})
	and are approximately constant.
	The probabilities of transitions~\ref{pm/equal},
	\ref{equal/pm}, and \ref{equal/equal},
	which impact the imbalance between classes
	but not the total queue lengths,
	vary with $\rho$.
	In particular, the probability
	of transitions~\ref{equal/equal}
	is maximal when $\rho = \frac12$,
	which may
	explain why $\rho = \frac12$
	minimizes the average waiting probability
	and mean waiting time.

	% bibliography
%	\bibliographystyle{plain}
%	\bibliography{paper}

	\appendix

	\section*{Appendix: Proofs of the results
		of Section~\ref{sec:performance}} \label{app:performance}
	
	\def\proofname{Proof of Corollary~\ref{coro:busy-sequence}}
	\begin{proof}
		Summing~\eqref{eq:piAB-proof}
		over all $\A \in \ind$
		and rearranging the sum symbols yields
		\begin{align}
			\label{eq:busy-sequence-1}
			& \sum_{(i, k) \in \I \times \K}
			\lambda_i \mu_k
			\sum_{\substack{
					\A \in \ind:
					i \in \I(\A \cap \K), \\
					k \in \K(\A \cap \I)
			}}
			\pi(\A)
			= \sum_{\substack{
					(i, k) \in \I \times \K: \\
					i \nsim k
			}}
			\lambda_i \mu_k
			\sum_{\substack{
					\A \in \ind: \\
					i \in \A, k \in \A
			}} \big(
			\pi(\A)
			+ \pi(\A {\setminus} \{i\})
			+ \pi(\A {\setminus} \{k\})
			+ \pi(\A {\setminus} \{i, k\})
			\big).
		\end{align}
		The left-hand side of this equation
		is the left-hand side of~\eqref{eq:busy-sequence}.
		The right-hand side can be rewritten
		by making changes of variables.
		For instance,
		for each $i \in \I$ and $k \in \K$
		such that $i \nsim k$,
		replacing $\A$ with $\A {\setminus} \{i, k\}$
		in the last sum yields
		\begin{align*}
			\sum_{\substack{
					\A \in \ind:
					i \in \A, k \in \A
			}} \pi(\A {\setminus} \{i, k\})
			&= \sum_{\substack{
					\A \subseteq \I \cup \K:
					i \notin \A, k \notin \A, \\
					\A \cup \{i, k\} \in \ind
			}} \pi(\A)
			= \sum_{\substack{
					\A \in \inde:
					i \notin \A, k \notin \A, \\
					i \notin \I(\A \cap \K),
					k \notin \K(\A \cap \I)
			}} \pi(\A).
		\end{align*}
		The second equality
		is true only because $i \nsim k$.
		By applying changes of variables
		to the other terms,
		we obtain that the right-hand side
		of~\eqref{eq:busy-sequence-1} is equal to
		\begin{align*}
			&\sum_{\substack{
					(i, k) \in \I \times \K: \\
					i \nsim k
			}}
			\lambda_i \mu_k
			\Bigg(
			\begin{aligned}[t]
				&\sum_{\substack{
						\A \in \inde:
						i \in \I, k \in \A, \\
						i \notin \I(\A \cap \K),
						k \notin \K(\A \cap \I)
				}}
				\pi(\A)
				+ \sum_{\substack{
						\A \in \inde:
						i \notin \I, k \in \A, \\
						i \notin \I(\A \cap \K),
						k \notin \K(\A \cap \I)
				}}
				\pi(\A) \\
				&+ \sum_{\substack{
						\A \in \inde:
						i \in \I, k \notin \A, \\
						i \notin \I(\A \cap \K),
						k \notin \K(\A \cap \I)
				}}
				\pi(\A)
				+ \sum_{\substack{
						\A \in \inde:
						i \notin \A, k \notin \A, \\
						i \notin \I(\A \cap \K),
						k \notin \K(\A \cap \I)
				}} \pi(\A)
				\Bigg)
			\end{aligned} \\
			&= \sum_{\substack{
						(i, k) \in \I \times \K:
						i \nsim k
				}}
				\lambda_i \mu_k
				\sum_{\substack{
						\A \in \inde:
						i \notin \I(\A \cap \K),
						k \notin \K(\A \cap \I)
				}} \pi(\A).
		\end{align*}
	\end{proof}
	
	\def\proofname{Proof of Proposition~\ref{prop:LipiAB}}
	\begin{proof}
		Let $i \in \I$.
		We have $L_i = \sum_{\A \in \inde} \ell_i(\A)$, where
		\begin{align*}
			\ell_i(\A)
			= \sum_{(c, d) \in \Pi_\A}
			|c|_i \pi(c, d),
			\quad \A \in \inde.
		\end{align*}
		Let $\A \in \ind$.
		If $i \notin \A$, we have directly
		$\ell_i(\A) = 0$
		because $|c|_i = 0$
		for each $(c, d) \in \Pi_\A$.
		Now assume that $i \in \A$
		(so that in particular $\A$ is non-empty).
		The method is similar to
		the proof of Proposition~\ref{prop:piAB}.
		First, by applying~\eqref{eq:picd}, we have
		\begin{align*}
			\ell_i(\A)
			= \sum_{(c, d) \in \Pi_\A}
			& |c|_i
				\frac{\lambda_{c_n}}
				{\mu(\K(\A \cap \I))}
				\frac{\mu_{d_n}}
				{\lambda(\I(\A \cap \K))}
				\pi((c_1, \ldots, c_{n-1}),
				(d_1, \ldots, d_{n-1})).
		\end{align*}
		Then applying~\eqref{eq:recurrence}
		and doing a change of variable yields
		\begin{align*}
			&\mu(\K(\A \cap \I)) \lambda(\I(\A \cap \K))
			\ell_i(\A) \\
			&=
			\begin{aligned}[t]
				&\lambda_i \sum_{k \in \A \cap \K} \mu_k
				\begin{aligned}[t]
					\bigg(
					&\sum_{(c, d) \in \Pi_\A}
					(|c|_i + 1) \pi(c, d)
					+ \sum_{(c, d) \in \Pi_{\A {\setminus} \{i\}}}
					(0 + 1) \pi(c, d) \\
					&+ \sum_{(c, d) \in \Pi_{\A {\setminus} \{k\}}}
					(|c|_i + 1) \pi(c, d)
					+ \sum_{(c, d) \in \Pi_{\A {\setminus} \{i, k\}}}
					(0 + 1) \pi(c, d)
					\bigg),
				\end{aligned} \\
				&+ \sum_{j \in (\A {\setminus} \{i\}) \cap \I}
				\sum_{k \in \A \cap \K}
				\lambda_j \mu_k
				\begin{aligned}[t]
					\bigg(
					&\sum_{(c, d) \in \Pi_\A}
					|c|_i \pi(c, d)
					+ \sum_{(c, d) \in \Pi_{\A {\setminus} \{j\}}}
					|c|_i \pi(c, d) \\
					&+ \sum_{(c, d) \in \Pi_{\A {\setminus} \{k\}}}
					|c|_i \pi(c, d)
					+ \sum_{(c, d) \in
						\Pi_{\A {\setminus} \{j, k\}}}
					|c|_i \pi(c, d)
					\bigg),
				\end{aligned}
			\end{aligned} \\
			&=
			\begin{aligned}[t]
				&\lambda_i \sum_{k \in \A \cap \K} \mu_k \left(
				\ell_i(\A)
				+ \pi(\A)
				+ \pi(\A {\setminus} \{i\})
				+ \ell_i(\A {\setminus} \{k\})
				+ \pi(\A {\setminus} \{k\})
				+ \pi(\A {\setminus} \{i, k\})
				\right) \\
				&+ \sum_{j \in (\A {\setminus} \{i\}) \cap \I}
				\sum_{k \in \A \cap \K}
				\lambda_j \mu_k \left(
				\ell_i(\A)
				+ \ell_i(\A {\setminus} \{j\})
				+ \ell_i(\A {\setminus} \{k\})
				+ \ell_i(\A {\setminus} \{j, k\})
				\right).
			\end{aligned}
		\end{align*}
		The result follows by rearranging the terms.
	\end{proof}
	
	\def\proofname{Proof of Proposition~\ref{prop:LIpiAB}}
	\begin{proof}
		Equation~\eqref{eq:LIpiAB} follows by
		summing~\eqref{eq:LipiAB} over all $i \in \I \cap \A$
		and simplifying the result
		using~\eqref{eq:piAB-proof}.
	\end{proof}
	
\end{document}